\documentclass{amsart}
\usepackage{amssymb, amsmath, amsthm, mathrsfs}
\usepackage{cite, hyperref}
\makeatletter
\renewcommand\@makefntext[1]{\noindent #1}
\makeatother

\usepackage[headsep=0.25truein, top=1.25truein, bottom=1truein, hmargin={1truein,1truein}]{geometry}
\allowdisplaybreaks[4]
\numberwithin{equation}{section}

\usepackage{graphicx,xcolor}
\definecolor{db}{RGB}{23,20,219}
\definecolor{dg}{RGB}{2,101,15}
\colorlet{sectitlecolor}{red!60!black}
\colorlet{sectboxcolor}{cyan!30}
\colorlet{secnumcolor}{orange}
\usepackage[explicit]{titlesec}
\titleformat{\section}
{\sffamily\color{sectitlecolor}\Large\bfseries\filcenter}{}{2em}{\thesection.\quad #1}%
\newtheoremstyle{mytheorem}{5pt}{}{\color{db}}{}{\color{db}\bfseries}{}{ }{}
\theoremstyle{mytheorem}
\newtheorem{theorem}{Theorem}[section]
\newtheorem{corollary}[theorem]{Corollary}

\newtheorem{proposition}[theorem]{Proposition}
\theoremstyle{definition}

\theoremstyle{example}

\theoremstyle{remark}
\newtheorem{remark}[theorem]{Remark}
\numberwithin{equation}{section}

\newcommand{\FTS}[6]{%
 {}_{3}F_{2}\!\!\left[\left.%
 \begin{matrix}%
 {#1}, {#2}, {#3} \vspace{1mm}\\%
 {#4}, {#5}%
 \end{matrix} \ \right\vert \ {#6} \right]} 
\newcommand{\ov}[1]{\overline{#1}}
\newcommand{\zetabar}[1]{\zeta(\overline{#1})}
\newcommand{\bm}[1]{\boldsymbol{#1}}
\def\Li{{\textrm{Li}}}

\begin{document}
\footnotetext{%
Date: 2021-10-01; Version 7. 
}

\title{On the convolutions of sums of multiple zeta(-star) values of height one}

\author{Kwang-Wu Chen}
\address{Department of Mathematics, University of Taipei, 10048 Taipei, Taiwan}
\email{kwchen@utaipei.edu.tw}
\thanks{The first author (corresponding author)
was funded by the Ministry of Science and Technology,
Taiwan, R.O.C., under Grant MOST 110-2115-M-845-001.}

\author{Minking Eie}
\address{Department of Mathematics, National Chung Cheng University, 168 University Road, Min-Hsiung, Chia-Yi 62145, Taiwan}
\email{minkingeie@gmail.com}


\begin{abstract}In this paper, we investigate  
the sums of mutliple zeta(-star) values of height one:
$Z_{\pm}(n)=\sum_{a+b=n}
(\pm 1)^b\zeta(\{1\}^a,b+2)$,
$Z_{\pm}^{\star}(n)=\sum_{a+b=n}
(\pm 1)^b\zeta^{\star}(\{1\}^a,b+2)$.
In particular, we prove that the weighted sum 
$\sum_{\substack{0\leq m\leq p\\ m: {\rm even}}}
\sum_{\mid\bm\alpha\mid=p+3}
2^{\alpha_{m+1}+1}\zeta(\alpha_0,\alpha_1,\ldots,\alpha_m,\alpha_{m+1}+1)
$
can be evaluated through the convolution of $Z_{-}(m)$ and $Z_{+}(n)$ with $m+n=p$.
\end{abstract}



\maketitle

\section{Introduction}\label{sec1}
For an $r$-tuple $\boldsymbol{\alpha} = (\alpha_{1}, \alpha_{2}, \ldots, \alpha_{r})$ 
of positive integers with $\alpha_{r} \geq 2$, a multiple zeta value $\zeta(\bm\alpha)$
and a multiple zeta-star value $\zeta^\star(\bm\alpha)$
are defined to be \cite{E09, E13, Ohno2005}
$$
  \zeta(\boldsymbol{\alpha})
  = \sum_{1 \leq k_{1} < k_{2} < \cdots < k_{r}} k_{1}^{-\alpha_{1}}
    k_{2}^{-\alpha_{2}} \cdots k_{r}^{-\alpha_{r}},
$$
and
$$
  \zeta^{\star}(\boldsymbol{\alpha})
  = \sum_{1 \leq k_{1} \leq k_{2} \leq \cdots \leq k_{r}} k_{1}^{-\alpha_{1}}
    k_{2}^{-\alpha_{2}} \cdots k_{r}^{-\alpha_{r}}.
$$
We denote the parameters $w(\bm\alpha)=\mid\bm\alpha\mid 
= \alpha_{1} + \alpha_{2} + \cdots + \alpha_{r}$,
$d(\bm\alpha)=r$, and $h(\bm\alpha)=\#\{i\mid \alpha_i>1, 1\leq i\leq r\}$,
called respectively the weight, the depth, and the height of $\bm\alpha$
(or of $\zeta(\bm\alpha)$, or of $\zeta^\star(\bm\alpha)$).
Multiple zeta values of height one 
are of the form $\zeta(\{1\}^{m},n+2)$,
where $\{a\}^{k}$ is $k$ repetitions of $a$. 
They have the generating function \cite{BBB1997, OZ01}

\begin{align}\label{eq.gn1}
  \sum_{m=0}^{\infty} \sum_{n=0}^{\infty} \zeta(\{1\}^{m},n+2) x^{m+1} y^{n+1}
  &= 1 - \frac{\Gamma(1-x) \Gamma(1-y)}{\Gamma(1-x-y)} \\
  &= 1 - \exp \left\{ \sum_{k=2}^{\infty} (x^{k}+y^{k}-(x+y)^{k})
    \frac{\zeta(k)}{k} \right\}.\nonumber
\end{align}

On the other hand, we derive the generating function of multiple zeta-star values
of height one $\zeta^\star(\{1\}^m,n+2)$ as follows.

\begin{equation}\label{eq.gn2}
\sum^\infty_{m=0}\sum^\infty_{n=0}\zeta^\star(\{1\}^m,n+2)x^my^n
=\frac{1}{(1-x)(1-y)}
\FTS{1}{1}{1-y}{2-x}{2-y}{1},
\end{equation}
where ${}_3F_2$ is the generalization hypergeometry series which is defined by
$$
\FTS{a_1}{a_2}{a_3}{b_1}{b_2}{z}
=\sum^\infty_{k=0}\frac{(a_1)_k(a_2)_k(a_3)_k}{(b_1)_k(b_2)_k}\frac{z^k}{k!}
$$
The generating functions of multiple zeta-star values for any fixed weight, depth, and height 
can be found in \cite{AKO2007}.
For any multiple zeta value $\zeta(\bm\alpha)$, 
we put a bar on top of $\alpha_j$ ($j=1,2,\ldots,r$) if there is a sign
$(-1)^{k_j}$ appearing in the denominator of its summation \cite{BBB1997, Xu2019}. 
For example,
$$
\zeta(\ov{\alpha_1},\ov{\alpha_2},\alpha_3,\ldots,\alpha_r)
=\sum_{1\leq k_1<k_2<\cdots<k_r}
\frac{(-1)^{k_1+k_2}}
{k_1^{\alpha_1}k_2^{\alpha_2}k_3^{\alpha_3}\cdots k_r^{\alpha_r}}.
$$
Due to Kontsevich \cite{LM95}, multiple zeta values can be expressed as iterated integrals over simplices of weight dimension:
\begin{equation} \label{e1.1}
  \zeta(\alpha_{1}, \alpha_{2}, \ldots, \alpha_{r})
  = \int_{E_{\mid\boldsymbol{\alpha}\mid}} \Omega_{1} \Omega_{2} \cdots
    \Omega_{\mid\boldsymbol{\alpha}\mid} \quad \textrm{or} \quad
    \int_{0}^{1} \Omega_{1} \Omega_{2} \cdots \Omega_{\mid\boldsymbol{\alpha}\mid}
\end{equation}
with $E_{\mid\boldsymbol{\alpha}\mid}: 0 < t_{1} < t_{2} < \cdots
< t_{\mid\boldsymbol{\alpha}\mid} < 1$ and
\begin{equation}\label{eq.domain}
  \Omega_{j} = \begin{cases}
  \dfrac{\mathrm{d}t_{j}}{1-t_{j}}
    &\textrm{if $j = 1, \alpha_{1}+1, \alpha_{1}+\alpha_{2}+1, \ldots, \alpha_{1}+\alpha_{2}+\cdots+\alpha_{r-1}+1$}, \\
  \dfrac{\mathrm{d}t_{j}}{t_{j}} &\textrm{otherwise}. \\
  \end{cases}
\end{equation}
For example,
\[
  \zeta(p+2)
  = \int_{E_{p+2}} \frac{\mathrm{d}t_{1}}{1-t_{1}} \prod_{k=2}^{p+2}
    \frac{\mathrm{d}t_{k}}{t_{k}}
\]
and 
\[
  \zeta(\{1\}^{m},n+2)
  = \int_{E_{m+n+2}} \prod_{j=1}^{m+1} \frac{\mathrm{d}t_{j}}{1-t_{j}}
    \prod_{k=m+2}^{m+n+2} \frac{\mathrm{d}t_{k}}{t_{k}},
\]
where $p$, $m$ and $n$ are all nonnegative integers. When a multiple zeta value $\zeta(\boldsymbol{\alpha})$ is expressed in iterated integrals as shown in \eqref{e1.1}, its \emph{dual}, is obtained by a change of variables
\[
  u_{1} = 1 - t_{\mid\boldsymbol{\alpha}\mid}, \quad
  u_{2} = 1 - t_{\mid\boldsymbol{\alpha}\mid-1}, \quad \ldots, \quad
  u_{\mid\boldsymbol{\alpha}\mid} = 1 - t_{1}.
\]
We list two simple dual results here
\[
  \zeta(p+2) = \zeta(\{1\}^{p},2) \quad \textrm{and} \quad
  \zeta(\{1\}^{m},n+2) = \zeta(\{1\}^{n},m+2),
\]
where $p$, $m$, $n$ are all nonnegative integers.
In 2016, Kaneko and Sakata \cite{KS2016} gave an explicit formula for the 
MZV of height one: For any integers $r,k\geq 1$,
$$
\zeta(\{1\}^{r-1},k+1)=\sum^{\min(r,k)}_{j=1}(-1)^{j-1}
\sum_{w(\bm a)=k, w(\bm b)=r\atop d(\bm a)=d(\bm b)=j}
\zeta(\bm{a+b}),
$$
where, for $\bm a=(a_1,\ldots,a_j)$ and $\bm b=(b_1,\ldots,b_j)$ of the same depth,
$\zeta(\bm{a+b})$ denotes $\zeta(a_1+b_1,\ldots,a_j+b_j)$.

A special case of results proved by Le and Murakami \cite{LM95} is stated as follows.
\begin{equation}\label{eq.a.1}
\sum_{a+b=2w-2}(-1)^{a+1}\zeta(\{1\}^a,b+2)
=2\zeta(\ov{2w}).
\end{equation}
Arakawa and Kaneko \cite{AK1999} defined the function
$$
\xi_k(s)=\frac1{\Gamma(s)}\int^\infty_0\frac{t^{s-1}}{e^t-1}
\Li_k(1-e^{-t})\,dt,
$$
where $\Li_k(s)$ denotes the $k$-th polylogarithm 
$\Li_k(s)=\sum^\infty_{n=1}\frac{s^n}{n^k}$.
It is exactly the multiple zeta-star values of height one
$$
\xi_k(s)=\zeta^\star(\{1\}^{s-1},k+1).
$$
They gave a representation as follows \cite[Theorem 6.(ii)]{AK1999}.
\begin{multline*}
\xi_k(s) = (-1)^{k-1}\Bigl(\sum_{a+b=k-2}\zeta(\{1\}^a,2,\{1\}^b,s)
+s\zeta(\{1\}^{k-1},s+1)\Bigr) 
+\sum^{k-2}_{j=0}(-1)^j\zeta(k-j)\zeta(\{1\}^j,s).
\end{multline*}
Recently, there are a lot of properties of generalized Arakawa-Kaneko zeta functions
discovered (ref. \cite{Chen2019, CC2010, Kargein2020}). 
In 2005, Ohno \cite[Theorem 8]{Ohno2005} proved that for any integer $w>1$,
\begin{equation}\label{eq.d.1}
\sum_{a+b=w-2}\zeta^\star(\{1\}^a,b+2)
=2(w-1)(1-2^{1-w})\zeta(w)=-2(w-1)\zeta(\ov{w}).
\end{equation}
In this paper, we define
\begin{alignat*}{3}
  Z_-(n)&= \sum_{a+b=n} (-1)^{b} \zeta(\{1\}^{a},b+2),
  &\quad\textrm{ and } \quad
  &&Z_+(n)&= \sum_{a+b=n} \zeta(\{1\}^{a},b+2),\\
  Z^\star_-(n)&=\sum_{a+b=n} (-1)^b \zeta^\star(\{1\}^a,b+2),
  &\quad\textrm{ and } \quad
  &&Z^\star_+(n)&= \sum_{a+b=n} \zeta^\star(\{1\}^a,b+2).
\end{alignat*}
Eq\,(\ref{eq.a.1}) and Eq\,(\ref{eq.d.1}) are properties of $Z_-(n)$ and $Z^\star_+(n)$.
Thus, we want to investigate the properties and their values 
of these four functions.

Here we list our main theorems as follows:

\begin{theorem} \label{T1.1}
For any nonnegative integer $p$, we have
\begin{align}\label{eq.abw}\nonumber
  &\sum_{\substack{0 \leq m \leq p \\ m: \textup{ even}}}
    \sum_{\mid\boldsymbol{\alpha}\mid=p+3} 2^{\alpha_{m+1}+1}
    \zeta(\alpha_{0}, \alpha_{1}, \ldots, \alpha_{m}, \alpha_{m+1}+1) \\
  &\quad= \sum_{m=0}^{\lfloor p/2\rfloor} Z_-(2m) Z_+(p-2m) 
    + Z_+(p+2)
    +\left\{\begin{array}{ll}
    (p+2)\zeta(p+4)-Z_-(p+2), &\mbox{ if $p$ is even,}\\
    (p-1)\zeta(p+4), &\mbox{ if $p$ is odd.}
    \end{array}\right. 
\end{align}
\end{theorem}

\begin{theorem}\label{thm.2}
For any nonnegative integer $p$, we have
\begin{equation}\label{eq.adw}
\sum^p_{m=0}Z_-(2m)Z^\star_+(2p-2m)
=(2p+2)(2p+1+2^{-(p+2)})\zeta(2p+4).
\end{equation}
\end{theorem}

On the other hand, we consider another integral
\[
  \frac{1}{p!} \int_{E_{2} \times E_{2}}
    \left( \log \frac{t_{2}}{t_{1}} + \lambda \log \frac{u_{2}}{u_{1}} \right)^{p} \frac{\mathrm{d}t_{1} \mathrm{d}t_{2}}{(1-t_{1})t_{2}} \frac{\mathrm{d}u_{1} \mathrm{d}u_{2}}{(1-u_{1})u_{2}}.
\]
Indeed, this integral is the integral representation of the weighted convolution of 
$$
\sum_{m+n=p}\lambda^n\zeta(m+2)\zeta(n+2).
$$
Under the same procedure, we obtain a weighted sum formula on Euler double sums.

\begin{theorem} \label{T1.2}
Let $p$ be a nonnegative integer and $\lambda$ a real number with $\lambda \neq 1$. Then
\begin{equation}\label{eq.l.2}
  \sum_{\alpha_{1}+\alpha_{2}=p+3}
    \big[ (\lambda^{\alpha_{1}-1}+1) (\lambda+1)^{\alpha_{2}}
    - \lambda^{\alpha_{1}-1} - \lambda^{\alpha_{2}} \big] \zeta(\alpha_{1},\alpha_{2}+1)
  = \frac{\lambda^{p+3}-1}{\lambda-1} \zeta(p+4).
\end{equation}
\end{theorem}

When $\lambda = 0$, Theorem~\ref{T1.2} gives the sum formula \cite{G96}
\[
  \sum_{\alpha_{1}+\alpha_{2}=p+3} \zeta(\alpha_{1},\alpha_{2}+1)
  = \zeta(p+4),
\]
and it is easy to see that
\[
  \sum_{\alpha_{1}+\alpha_{2}=p+3} 2^{\alpha_{2}}
    \zeta(\alpha_{1},\alpha_{2}+1)
  = \frac{p+5}{2} \zeta(p+4)
\]
when $\lambda \to 1$. This weighted sum formula was first derived by Ohno and Zudilin \cite{OZ08} in 2008. Let $\lambda=-1$, this gives another interesting formula
\begin{equation}
\sum_{\alpha_1+\alpha_2=2p+3}(-1)^{\alpha_1}\zeta(\alpha_1,\alpha_2+1)
=\frac12\zeta(2p+4).
\end{equation}

This paper is organized as follows: In Section~\ref{S2}, 
we give some preliminaries about the shuffle product formulas and 
the stuffle relations on the multiple zeta(-star) values and the 
alternating mutiple zeta values.  We also list some propositions 
that will be used frequently later. In the next section, we gives the evaluations 
of $Z_{\pm}(n)$ and $Z^\star_{\pm}(n)$.
We also give the generating function of $\zeta^\star(\{1\}^m,n+2)$
in Section~\ref{sec.3}. We prove Theorem~\ref{T1.1} in Section~\ref{S4},
and Theorem~\ref{thm.2} in Section~\ref{sec.5}.
We also prove a simple weighted sum formula for an alternating Euler double sums:
$$
\sum_{a+b=n}2^{b+1}\zeta(a+1,\ov{b+1})
=(n+1)\zeta(n+2)+2\zeta(n+1,\ov{1})+2\zeta(\ov{n+2}).
$$
We give a proof of Theorem~\ref{T1.2} in the final section.
\vspace{.5truecm}
\section{Some preliminaries and auxiliary tools} \label{S2}
For the arguments of multiple zeta values, we hereafter use the lowercase English letters and lowercase Greek letters, with or without subscripts, to denote the nonnegative and positive integers, unless otherwise specified. For instance,
\[
  \sum_{a+b=3} (-1)^{b} \zeta(\{1\}^{a},b+2)
  = -\zeta(5) + \zeta(1,4) - \zeta(1,1,3) + \zeta(1,1,1,2)
\]
and
\[
  \sum_{\mid\boldsymbol{\alpha}\mid=4} \zeta(\alpha_{1}, \alpha_{2}, \alpha_{3}+1)
  = \zeta(1,1,3) + \zeta(1,2,2) + \zeta(2,1,2).
\]

The shuffle product formula of two multiple zeta values is defined as
\[
  \int_{0}^{1} \Omega_{1} \Omega_{2} \cdots \Omega_{m}
    \int_{0}^{1} \Omega_{m+1} \Omega_{m+2} \cdots \Omega_{m+n}
  = \sum_{\sigma} \int_{0}^{1} \Omega_{\sigma(1)} \Omega_{\sigma(2)} \cdots
    \Omega_{\sigma(m+n)},
\]
where the sum is taken over all $\binom{m+n}{m}$ permutations $\sigma$ on the set $\{ 1,2,\ldots,m+n \}$ which preserve the orders of $\Omega_{1} \Omega_{2} \cdots \Omega_{m}$ and $\Omega_{m+1} \Omega_{m+2} \cdots \Omega_{m+n}$. More precisely, for all $1 \leq i < j \leq m$ and $m+1 \leq i < j \leq m+n$, we have
\[
  \sigma^{-1}(i) < \sigma^{-1}(j).
\]

Indeed, it is a hard job to perform shuffle products of multiple zeta values owing to number of variables in their iterated integrals maybe very large. The shuffle product of two multiple zeta values of weight $m$ and $n$ is equivalent to finding all possible interlacings of two sets of variables
\[
  0 < t_{1} < t_{2} < \cdots < t_{m} < 1 \quad \textrm{and} \quad
  0 < u_{1} < u_{2} < \cdots < u_{n} < 1.
\]
In the case of the shuffle product of $\zeta(2)$ and $\zeta(2)$, we have to find all interlacings of
\[
  0 < t_{1} < t_{2} < 1 \quad \textrm{and} \quad 0 < u_{1} < u_{2} < 1.
\]
There are six simplices of dimension $4$:
\begin{gather*}
  D_{1}: 0 < t_{1} < t_{2} < u_{1} < u_{2} < 1, \quad
  D_{2}: 0 < u_{1} < u_{2} < t_{1} < t_{2} < 1, \\
  D_{3}: 0 < t_{1} < u_{1} < t_{2} < u_{2} < 1, \quad
  D_{4}: 0 < t_{1} < u_{1} < u_{2} < t_{2} < 1, \\
  D_{5}: 0 < u_{1} < t_{1} < u_{2} < t_{2} < 1 \quad \textrm{and} \quad
  D_{6}: 0 < u_{1} < t_{1} < t_{2} < u_{2} < 1.
\end{gather*}
They produce $\zeta(2,2)$ on $D_{1}$ and $D_{2}$, and produce $\zeta(1,3)$ on $D_{3}, D_{4}, D_{5}, D_{6}$. So the resulting shuffle relation is
\[
  \zeta(2) \zeta(2) = 2\zeta(2,2) + 4\zeta(1,3).
\]

To overcome the difficulty of shuffle products, we first reduce the number of variables in the iterated integrals of multiple zeta values. Here we list some known results that will be used in subsequent sections.

\begin{proposition} \textup{\cite{ELO09}} \label{P2.1}
For any nonnegative integer $p$, we have
\[
  \zeta(p+2)
  = \frac{1}{(p+1)!} \int_{0}^{1} \left( \log \frac{1}{t} \right)^{p+1}
    \frac{\mathrm{d}t}{1-t}.
\]
\end{proposition}

If we change the integral operators $\Omega_j$ in Eq.\,(\ref{eq.domain}) as follows.
\begin{equation}
  \Omega_{j} = \begin{cases}
  \dfrac{-\mathrm{d}t_{j}}{1+t_{j}}
    &\textrm{if $j = 1, \alpha_{1}+1, \alpha_{1}+\alpha_{2}+1, \ldots, \alpha_{1}+\alpha_{2}+\cdots+\alpha_{r-1}+1$}, \\
  \dfrac{\mathrm{d}t_{j}}{t_{j}} &\textrm{otherwise}. \\
  \end{cases}
\end{equation}
The corresponding alternating multiple zeta value is represented as
\begin{equation}\label{eq.bar}
\int_{E_{\mid\bm\alpha\mid}}\Omega_1\Omega_2\cdots \Omega_{\mid\bm\alpha\mid}
=\zeta(\alpha_1,\ldots,\alpha_{r-1},\ov{\alpha_r}).
\end{equation}
Then we have a similar result on alternating Euler sums.
\begin{equation}\label{eq.aeuler}
\zeta(\ov{p+1})=\frac{-1}{p!}\int^1_0\left(\log\frac1t\right)^p\frac{dt}{1+t}.
\end{equation}
In fact, this integral representation of $\zeta(\ov{p+1})$ and the following representation
of $\zeta(\{1\}^m,n+1)$ also can derived from 
\cite[Theorem 2.1]{Xu2019}.
\begin{equation}
\zeta(\{1\}^m,\ov{n+1})
=\frac{(-1)^{m+n+1}}{n!m!}\int^1_0\log^m(1+t)\log^n(t)\frac{dt}{1+t}.
\end{equation}

\begin{proposition} \textup{\cite{ELO09}} \label{P2.2}
For a pair of nonnegative integers $m$ and $n$, we have
\[
  \zeta(\{1\}^{m},n+2)
  = \frac{1}{m!n!} \int_{0 < t_{1} < t_{2} < 1}
    \left( \log \frac{1}{1-t_{1}} \right)^{m}
    \left( \log \frac{t_{2}}{t_{1}} \right)^{n}
    \frac{\mathrm{d}t_{1} \mathrm{d}t_{2}}{(1-t_{1})t_{2}}.
\]
\end{proposition}

\begin{proposition} \textup{\cite{ELO09}} \label{P2.3}
For nonnegative integers $b_{1}, b_{2}, \ldots, b_{r}, b_{r+1}$, we have
\begin{multline*}
  \zeta(b_{1}+1, b_{2}+1, \ldots, b_{r}+1, b_{r+1}+2)
  = \frac{1}{b_{1}! b_{2}! \cdots b_{r+1}!} \int_{E_{r+2}} \prod_{j=1}^{r+1}
    \frac{\mathrm{d}t_{j}}{1-t_{j}} \left( \log \frac{t_{j+1}}{t_{j}} \right)^{b_{j}} \frac{\mathrm{d}t_{r+2}}{t_{r+2}},
\end{multline*}
where $E_{r+2}$ is a simplex defined as $0 < t_{1} < t_{2} < \cdots < t_{r+2}$.
\end{proposition}

\begin{proposition} \textup{\cite{ELO09}} \label{P2.4}
For nonnegative integers $p$, $q$, $r$ and $\ell$, we have
\begin{align*}
  &\sum_{\mid\boldsymbol{\alpha}\mid=q+r+1}
    \zeta(\{1\}^{p}, \alpha_{1}, \alpha_{2}, \ldots, \alpha_{q}, \alpha_{q+1}+\ell+1) \\
  &= \frac{1}{p!q!r!\ell!} \int_{0 < t_{1} < t_{2} < 1}
    \left( \log \frac{1}{1-t_{1}} \right)^{p}
    \left( \log \frac{1-t_{1}}{1-t_{2}} \right)^{q}
    \left( \log \frac{t_{2}}{t_{1}} \right)^{r}
    \left( \log \frac{1}{t_{2}} \right)^{\ell}
    \frac{\mathrm{d}t_{1} \mathrm{d}t_{2}}{(1-t_{1})t_{2}}.
\end{align*}
In particular,
$$
\sum_{\mid\boldsymbol{\alpha}\mid=q+r+1}
    \zeta(\alpha_{1}, \alpha_{2}, \ldots, \alpha_{q}, \alpha_{q+1}+1) 
= \frac{1}{q!r!} \int_{0 < t_{1} < t_{2} < 1}
    \left( \log \frac{1-t_{1}}{1-t_{2}} \right)^{q}
    \left( \log \frac{t_{2}}{t_{1}} \right)^{r}
    \frac{\mathrm{d}t_{1} \mathrm{d}t_{2}}{(1-t_{1})t_{2}}.
$$
\end{proposition}

The shuffle product of two Riemann zeta values can give 
the well-known Euler decomposition formula \cite{EW13}.

\begin{proposition}[Euler decomposition formula]\label{T2.5}
For a pair of positive integers $p$ and $q$, we have
\begin{equation}
  \zeta(p+1) \zeta(q+1)
  = \sum_{\alpha_{1}+\alpha_{2} = p+q+1}
    \left[ \binom{\alpha_{2}}{p} + \binom{\alpha_{2}}{q} \right] \zeta(\alpha_{1},\alpha_{2}+1).
\end{equation}
\end{proposition}

Here we illustrate our improved shuffle product by the shuffle product of 
two alternating Euler sums. The resulted formula can be found in \cite{Teo2018}.
\begin{proposition}
For any nonnegative integers $p$ and $q$, we have
\begin{equation}\label{eq.2barm}
\zeta(\ov{p+1})\zeta(\ov{q+1})
=\sum_{a_1+a_2=p+q}\left[\binom{a_2}{p}
+\binom{a_2}{q}\right]\zeta(a_1+1,\ov{a_2+1}).
\end{equation}
\end{proposition}
\begin{proof}
The integral representation of $\zeta(\ov{p+1})$ is (ref. Eq.\,(\ref{eq.aeuler}))
\begin{equation}
\zeta(\ov{p+1})=\frac{-1}{p!}\int^1_0\left(\log\frac1t\right)^p\frac{dt}{1+t}.
\end{equation}
We express the product $\zeta(\ov{p+1}) \zeta(\ov{q+1})$ as a double integral
\begin{equation} \label{e2.1}
  \zeta(\ov{p+1}) \zeta(\ov{q+1})
  = \frac{1}{p!q!} \int_{0}^{1} \!\! \int_{0}^{1}
    \left( \log \frac{1}{t} \right)^{p}
    \left( \log \frac{1}{u} \right)^{q} \frac{\mathrm{d}t}{1+t} \frac{\mathrm{d}u}{1+u}.
\end{equation}
Decompose the square $[0,1] \times [0,1]$ into two simplices
\[
  D_{1}: 0 < t < u < 1 \quad \textrm{and} \quad D_{2}: 0 < u < t < 1.
\]
On the simplex $D_{1}$, we rewrite
\begin{align*}
  \frac{1}{p!} \left( \log \frac{1}{t}
    \right)^{p}
  &= \frac{1}{p!} \left( \log \frac{u}{t}
    + \log \frac{1}{u} \right)^{p} \\
  &= \sum_{a+b=p} \frac{1}{a!b!} \left( \log \frac{u}{t}
    \right)^{a} \left( \log \frac{1}{u} \right)^{b}.
\end{align*}
So the integral \eqref{e2.1} over $D_{1}$ is
\[
  \sum_{a+b=p} \frac{1}{a!b!q!}
    \int_{0 < t < u < 1} \left( \log \frac{u}{t} \right)^{a}
    \left( \log \frac{1}{u} \right)^{b+q} \frac{\mathrm{d}t}{1+t} \frac{\mathrm{d}u}{1+u}.
\]

Based on Eq.\,(\ref{eq.bar}) and Proposition~\ref{P2.3}, 
the integral above can be transformed into multiple zeta values as
\[
  \sum_{a+b=p+1} \binom{b+q}{q} \zeta(a+1,\ov{b+q+1}),
\]
or, by setting $a_{1} = a$ and $a_{2} = b+q$,
\[
  \sum_{a_{1}+a_{2}=p+q+1} \binom{a_{2}}{q}
    \zeta(a_{1}+1,\ov{a_{2}+1}).
\]

On the other hand, the integral \eqref{e2.1} over $D_{2}$ can be obtained from above simply exchange $p$ and $q$. The value is
\[
  \sum_{a_{1}+a_{2}=p+q+1} \binom{a_{2}}{p}
    \zeta(a_{1}+1,\ov{a_{2}+1}),
\]
which leads to our assertion.
\end{proof}

An quick similar shuffle product of a Riemann zeta value and an alternating zeta value 
gives us that $\zeta(p+2)\zeta(\ov{q+1})$ can be represented by
\begin{equation}\label{eq.1barm}
\sum_{a_1+a_2=p+q+1}\binom{a_2}{q}\zeta(\ov{a_1+1},\ov{a_2+1})
+\sum_{a_1+a_2=p+q}\binom{a_2+1}{p+1}\zeta(\ov{a_1+1},a_2+2),
\end{equation}
where $p$ and $q$ are nonnegative integers.
The classical well-known double-stuffle relations are list as below \cite{Teo2018}.
\begin{align}
\zeta(r)\zeta(s)&=\zeta(r,s)+\zeta(s,r)+\zeta(r+s),\\ \label{eq.2.11}
\zeta(r)\zeta(\ov{s}) &=\zeta(r,\ov{s})+\zeta(\ov{s},r)+\zeta(\ov{r+s}),\\ \label{eq.2.12}
\zeta(\ov{r})\zeta(\ov{s}) &= \zeta(\ov{r},\ov{s})+\zeta(\ov{s},\ov{r})+\zeta(r+s).
\end{align}
As an immediate consequence of Euler decomposition formula, we reprove a weighted sum formula due to Ohno and Zudilin in 2008.

\begin{proposition} \textup{\cite{OZ08}} \label{prop.27}
For any nonnegative integer $r$, we have
\[
  \sum_{\alpha_{1}+\alpha_{2}=r+3} 2^{\alpha_{2}}
    \zeta(\alpha_{1},\alpha_{2}+1)
  = \frac{r+5}{2} \zeta(r+4).
\]
\end{proposition}
\begin{proof}
Taking $p = \ell$ and $q = r-\ell$ with $0 \leq \ell \leq r$ in Theorem~\ref{T2.5}, Euler decomposition formula reads as
\[
  \sum_{\alpha_{1}+\alpha_{2}=r+3}
    \left[ \binom{\alpha_{2}}{\ell+1} + \binom{\alpha_{2}}{r-\ell+1} \right] \zeta(\alpha_{1},\alpha_{2}+1)
  = \zeta(\ell+2) \zeta(r-\ell+2).
\]
Summing over all $0 \leq \ell \leq r$ and noting
\[
  \sum_{\ell=0}^{r} \binom{\alpha_{2}}{\ell+1}
  = \sum_{\ell=0}^{r} \binom{\alpha_{2}}{r-\ell+1},
\]
we obtain that
\begin{equation} \label{e2.2}
  \sum_{\alpha_{1}+\alpha_{2}=r+3} \sum_{\ell=0}^{r} \binom{\alpha_{2}}{\ell+1}
    \zeta(\alpha_{1},\alpha_{2}+1)
  = \frac{1}{2} \sum_{\ell=0}^{r} \zeta(\ell+2) \zeta(r-\ell+2).
\end{equation}
According to the range of $\alpha_{2}$, we easily get the inner sum
\[
  \sum_{\ell=0}^{r} \binom{\alpha_{2}}{\ell+1} = \begin{cases}
  2^{\alpha_{2}}-1 &\textrm{if $1 \leq \alpha_{2} \leq r+1$}, \\
  2^{\alpha_{2}}-2 &\textrm{if $\alpha_{2} = r+2$}.
  \end{cases}
\]
Therefore, the left-hand side on the relation \eqref{e2.2} is equal to
\[
  \sum_{\alpha_{1}+\alpha_{2}=r+3} (2^{\alpha_{2}}-1)
    \zeta(\alpha_{1},\alpha_{2}+1) - \zeta(1,r+3).
\]
Our assertion then follows from an evaluation of $\zeta(1,r+3)$, due to Euler \cite{E1775}, as shown
\[
  \zeta(1,r+3)
  = \frac{r+3}{2} \zeta(r+4)
    - \frac{1}{2} \sum_{\ell=0}^{r} \zeta(\ell+2) \zeta(r-\ell+2)
\]
and the classical sum formula
\[
  \sum_{\alpha_{1}+\alpha_{2}=r+3} \zeta(\alpha_{1},\alpha_{2}+1)
  = \zeta(r+4).
\]
\end{proof}

Similarly, if we use the same method on Eq.\,(\ref{eq.2barm}), then
we get 
\begin{equation}\label{eq.2w}
\sum_{a+b=r}\zeta(\ov{a+1})\zeta(\ov{b+1})
=\sum_{a+b=r}2^{b+1}\zeta(a+1,\ov{b+1}).
\end{equation}

In the final section, we shall develop a more general weighted sum formula of Euler double sums.

We derived an integral representation of $\zeta^\star(\{1\}^q,n+2)$ in 
\cite[Corollary 2.3]{CE2018}. Here we present a classical method based on 
the Bose-Einstein integral form of $\Li_q(u)$.
\begin{proposition}
Let $q$ and $n$ be nonnegative integers. Then
\begin{equation}
\zeta^\star(\{1\}^q,n+2)
=\frac1{q!n!}\int_{E_2}\left(\log\frac1{1-t_2}\right)^q\left(\log\frac{t_2}{t_1}\right)^n
\frac{dt_1dt_2}{(1-t_1)t_2}.
\end{equation}
\end{proposition}
\begin{proof}
Since from the form of Bose-Einstein integral
$$
\Li_q(u)=\int^\infty_0\frac{t^{q-1}}{e^t/u-1}\,dt
=\frac1{\Gamma(q)}\int^u_0\left(\log\frac{u}{y}\right)^{q-1}\frac{dy}{1-y},
$$
we have
\begin{align*}
\zeta^\star(\{1\}^q,n+2) &= \frac1{\Gamma(q+1)}\int^\infty_0\frac{t^q}{e^t-1}
\Li_{n+1}(1-e^{-t})\,dt\\
&=\frac1{q!}\int^1_0\left(\log\frac1{1-u}\right)^q\Li_{n+1}(u)\frac{du}{u} \\
&=\frac1{q!n!}\int_{0<y<u<1}\left(\log\frac1{1-u}\right)^q\left(\log\frac{u}{y}\right)^n
\frac{dydu}{(1-y)u}.
\end{align*}
\end{proof}
\vspace{.5truecm}
\section{Evaluations of $Z_{\pm}(n)$, and $Z^\star_{\pm}(n)$} \label{sec.3}
Recall that
\begin{alignat*}{3}
  Z_-(n)&= \sum_{a+b=n} (-1)^{b} \zeta(\{1\}^{a},b+2),
  &\quad\textrm{ and } \quad
  &&Z_+(n)&= \sum_{a+b=n} \zeta(\{1\}^{a},b+2),\\
  Z^\star_-(n)&=\sum_{a+b=n} (-1)^b \zeta^\star(\{1\}^a,b+2),
  &\quad\textrm{ and } \quad
  &&Z^\star_+(n)&= \sum_{a+b=n} \zeta^\star(\{1\}^a,b+2).
\end{alignat*}
All $Z_{\pm}(n)$, and $Z^\star_{\pm}(n)$ 
are sums of multiple zeta(-star) values of height one and can be expressed as double integrals
\begin{align*}
  Z_-(n)
  &= \frac{1}{n!} \int_{E_{2}}
    \left( \log \frac{1}{1-t_{1}} - \log \frac{t_{2}}{t_{1}} \right)^{n} \frac{\mathrm{d}t_{1} \mathrm{d}t_{2}}{(1-t_{1})t_{2}}, \\
  Z_+(n)
  &= \frac{1}{n!} \int_{E_{2}}
    \left( \log \frac{1}{1-t_{1}} + \log \frac{t_{2}}{t_{1}} \right)^{n} 
    \frac{\mathrm{d}t_{1} \mathrm{d}t_{2}}{(1-t_{1})t_{2}},\\
  Z^\star_-(n)
  &= \frac{1}{n!} \int_{E_{2}}
    \left( \log \frac{1}{1-t_{2}} - \log \frac{t_{2}}{t_{1}} \right)^{n} \frac{\mathrm{d}t_{1} \mathrm{d}t_{2}}{(1-t_{1})t_{2}}, \\
  Z^\star_+(n)
  &= \frac{1}{n!} \int_{E_{2}}
    \left( \log \frac{1}{1-t_{2}} + \log \frac{t_{2}}{t_{1}} \right)^{n} 
    \frac{\mathrm{d}t_{1} \mathrm{d}t_{2}}{(1-t_{1})t_{2}}.
\end{align*}
Through the duality theorem $\zeta(\{1\}^{a},b+2) = \zeta(\{1\}^{b},a+2)$, we see that 
$Z_-(2m+1) = 0$. Therefore, we combine these with Eq.\,(\ref{eq.a.1}), then we have

\begin{proposition} \label{P3.1}
For any nonnegative integer $m$, we have $Z_-(2m+1)=0$ and 
\[
  Z_-(2m)
  = \zeta^{\star}(\{2\}^{m+1})
  = 2 \left( 1 - \frac{1}{2^{2m+1}} \right) \zeta(2m+2)=-2\zeta(\ov{2m+2}).
\]
\end{proposition}

The function $Z_+(n)$ can be represented as a polynomial of $\zeta(\ov{s})$
over $\mathbb Q$.

\begin{proposition}\label{prop.b}
For any nonnegative integer $n$, we have
\begin{equation}
Z_+(n)=\sum_{a+b=n}\zeta(\{1\}^a,b+2)
=-2^{n+2}P_{n+2}(0,\zetabar{2},\ldots,\zetabar{n+2}),
\end{equation}
where 
$P_n(t_1,\ldots,t_n)$ is the modified Bell polynomials defined by 
\cite{Chen2017, CC2010}
\begin{equation}
\exp\left(\sum^\infty_{k=1}\frac{x_k}{k}z^k\right)
=\sum^\infty_{m=0}P_m(x_1,\ldots,x_m)z^m.
\end{equation}
\end{proposition}
\begin{proof}
We substitute $x=y$ in Eq.\,(\ref{eq.gn1}), we have
\begin{align*}
\sum^\infty_{n=0}Z_+(n)x^{n+2}
&=1-\exp\left\{\sum^\infty_{k=2}\frac{(2-2^k)\zeta(k)}{k}x^k\right\}
=1-\exp\left\{\sum^\infty_{k=2}\frac{\zeta(\ov{k})}{k}(2x)^k\right\}\\
&=1-\sum^\infty_{m=0}P_m(0,\zeta(\ov{2}),\zeta(\ov{3}),\ldots,\zeta(\ov{m}))(2x)^m.
\end{align*}
Comparing the coefficient of $x^{n+2}$, we have 
$$
Z_+(n)=-2^{n+2}P_{n+2}(0,\zetabar{2},\ldots,\zetabar{n+2}),
$$
for $n\geq 0$.
\end{proof}
Here we list some values of $Z_+(n)$ as follows.
\begin{align*}
Z_+(0) &= -2\zeta(\ov{2}) \ =\ \zeta(2),\\
Z_+(1) &= -\frac83\zeta(\ov{3}) \ =\ 2\zeta(3),\\
Z_+(2) &=-2\bigl(\zeta(\ov{2})^2+2\zeta(\ov{4})\bigr),\\
Z_+(3) &= -\frac{16}{15}\bigl(5\zeta(\ov{2})\zeta(\ov{3})+6\zeta(\ov{5})\bigr),\\
Z_+(4) &=-\frac49\bigl(3\zeta(\ov{2})^3+18\zeta(\ov{2})\zeta(\ov{4})
+8\zeta(\ov{3})^2+24\zeta(\ov{6})\bigr).
\end{align*}
\begin{proposition}
The generating function of $\zeta^\star(\{1\}^m,n+2)$ is 
$$
\sum^\infty_{m=0}\sum^\infty_{n=0}\zeta^\star(\{1\}^m,n+2)x^my^n
=\frac1{(1-x)(1-y)}\FTS{1}{1}{1-y}{2-x}{2-y}{1}.
$$
\end{proposition}
\begin{proof}
First, we write
$$
\sum^\infty_{m=0}\sum^\infty_{n=0}\zeta^\star(\{1\}^m,n+2)x^my^n
=\sum^\infty_{n=0}\sum^\infty_{k=1}
\prod^k_{j=1}\left(1-\frac{x}{j}\right)^{-1}
\frac{y^n}{k^{n+2}}.
$$
After summing on $n$, what remains is an instance of 
the hypergeometric series ${}_3F_2(1)$:
$$
\frac1{(1-x)(1-y)}\FTS{1}{1}{1-y}{2-x}{2-y}{1}.
$$
\end{proof}
If we set $y=0$ and compare the coefficients of $x^m$, we have \cite{G96, OZ08}
$$
\zeta^\star(\{1\}^m,2)=(m+1)\zeta(m+2).
$$

The following equation can be easily obtained by substituting $y=-x$ in
the generating function.
\begin{proposition}
For any nonnegative integer $n$, we have
\begin{equation}
Z^\star_-(n)=\sum_{a+b=n}(-1)^b\zeta^\star(\{1\}^a,b+2)
=\sum_{a+2b=n}\zeta^\star(a+2,\{2\}^b).
\end{equation}
\end{proposition}
We can also use a result in \cite[Eq.\,(11)]{CE2018}:
$$
\frac{1}{a!b!}\int_{E_2}\left(\log\frac{1-t_1}{1-t_2}\right)^a
\left(\log\frac1{1-t_1}-\log\frac{t_2}{t_1}\right)^b\frac{dt_1dt_2}{(1-t_1)t_2}
=\left\{\begin{array}{ll}\zeta^\star(a+2,\{2\}^m),&\mbox{ if }b=2m,\\
0,&\mbox{ if }b=2m+1.\end{array}\right.
$$
Summing over all possible nonnegative integers $a$ and $b$ 
with $a+b=n$, and we get the same
integral representation of $Z^\star_-(n)$.

We have list the values of $Z^\star_+(n)$ in Eq.\,(\ref{eq.d.1}), 
i.e. for any nonnegative integer $n$,
$$
Z^\star_+(n)=2(n+1)(1-2^{-n-1})\zeta(n+2)=-2(n+1)\zeta(\ov{n+2}).
$$
In the next section,
we will use the convolution of $Z_-(n)$ and $Z_+(m)$ 
to construct a weighted sum formula.
\vspace{.5truecm}
\section{Convolution of $Z_-(n)$ and $Z_+(m)$} \label{S4}
In this section, we describe our general procedure to the evaluation of the weighted sum
\[
  \sum_{\substack{0 \leq m \leq p \\ m: \textup{ even}}}
    \sum_{\mid\boldsymbol{\alpha}\mid=p+3} 2^{\alpha_{m+1}}
    \zeta(\alpha_{0}, \alpha_{1}, \ldots, \alpha_{m}, \alpha_{m+1}+1).
\]
\begin{theorem}
For any nonnegative integer $p$, we have
\begin{align}\nonumber
&  \sum_{\substack{0 \leq m \leq p \\ m: \textup{ even}}}
    \sum_{\mid\boldsymbol{\alpha}\mid=p+3} 2^{\alpha_{m+1}+1}
    \zeta(\alpha_{0}, \alpha_{1}, \ldots, \alpha_{m}, \alpha_{m+1}+1) \\
  &\quad= \sum_{m=0}^{\lfloor p/2\rfloor} Z_-(2m) Z_+(p-2m)
    + Z_+(p+2)
    +\left\{\begin{array}{ll}
    (p+2)\zeta(p+4)-Z_-(p+2), &\mbox{ if $p$ is even,}\\
    (p-1)\zeta(p+4), &\mbox{ if $p$ is odd.}
    \end{array}\right.
\end{align}
\end{theorem}
\begin{proof}
We begin with the integral
\begin{equation} \label{e4.1}
  \frac{1}{p!} \int_{E_{2} \times E_{2}}
    \left( \log \frac{1-t_{1}}{1-u_{1}} + \log \frac{t_{2}}{t_{1}}
    + \log \frac{u_{2}}{u_{1}} \right)^{p}
    \frac{\mathrm{d}t_{1} \mathrm{d}t_{2}}{(1-t_{1})t_{2}} \frac{\mathrm{d}u_{1} \mathrm{d}u_{2}}{(1-u_{1})u_{2}},
\end{equation}
where $p$ is a nonnegative integer and $E_{2}$ is a simplex in two-dimensional Euclidean space defined by $E_{2}: 0 < t_{1} < t_{2} < 1$ and $E_{2}: 0 < u_{1} < u_{2} < 1$. Note that the integral \eqref{e4.1} is separable and it can be expressed as a product of two sums of multiple zeta values of height one.

Expand the integrand as
\[
  \sum_{m+n=p} \frac{1}{m!n!}
    \left( -\log \frac{1}{1-t_{1}} + \log \frac{t_{2}}{t_{1}} \right)^{m} \left( \log \frac{1}{1-u_{1}} + \log \frac{u_{2}}{u_{1}} \right)^{n}.
\]
By Proposition~\ref{P2.2}, the value of integral \eqref{e4.1} is
\begin{equation} \label{e4.2}
  \sum_{m+n=p} (-1)^{m} Z_-(m) Z_+(n)
\end{equation}
with
\[
  Z_-(m)
  = \frac{1}{m!} \int_{E_{2}}
    \left( \log \frac{1}{1-t_{1}} - \log \frac{t_{2}}{t_{1}} \right)^{m} \frac{\mathrm{d}t_{1} \mathrm{d}t_{2}}{(1-t_{1})t_{2}}
  = \sum_{a+b=m} (-1)^{b} \zeta(\{1\}^{a},b+2)
\]
and
\[
  Z_+(n)
  = \frac{1}{n!} \int_{E_{2}}
    \left( \log \frac{1}{1-u_{1}} + \log \frac{u_{2}}{u_{1}} \right)^{n} \frac{\mathrm{d}u_{1} \mathrm{d}u_{2}}{(1-u_{1})u_{2}}
  = \sum_{c+d=n} \zeta(\{1\}^{c},d+2).
\]
Both $Z_-(m)$ and $Z_+(n)$ can be evaluated further in terms of single zeta values as in Propositions~\ref{P3.1} and \ref{prop.b}.

We now proceed to carry out the shuffle products of $Z_-(m)$ and $Z_+(n)$ 
with $m+n = p$. We decompose $E_{2} \times E_{2}$ into a disjoint union 
of $6$ simplices of dimension $4$:
\begin{gather*}
  D_{1}: 0 < t_{1} < t_{2} < u_{1} < u_{2} < 1, \quad
  D_{2}: 0 < u_{1} < u_{2} < t_{1} < t_{2} < 1, \\
  D_{3}: 0 < t_{1} < u_{1} < t_{2} < u_{2} < 1, \quad
  D_{4}: 0 < t_{1} < u_{1} < u_{2} < t_{2} < 1, \\
  D_{5}: 0 < u_{1} < t_{1} < u_{2} < t_{2} < 1 \quad \textrm{and} \quad
  D_{6}: 0 < u_{1} < t_{1} < t_{2} < u_{2} < 1.
\end{gather*}
For $j = 1,2,\ldots,6$, we let
\[
  I(p;j)
  = \frac{1}{p!} \int_{D_{j}}
    \left( \log \frac{1-t_{1}}{1-u_{1}} + \log \frac{t_{2}}{t_{1}}
    + \log \frac{u_{2}}{u_{1}} \right)^{p}
    \frac{\mathrm{d}t_{1} \mathrm{d}t_{2}}{(1-t_{1})t_{2}} \frac{\mathrm{d}u_{1} \mathrm{d}u_{2}}{(1-u_{1})u_{2}}.
\]
So we have the identity
\[
  \sum_{j=1}^{6} I(p;j) = \sum_{m+n = p} (-1)^{m} Z_-(m) Z_+(n).
\]
The remaining is to evaluate each $I(p;j)$ one by one in terms of multiple zeta values.

On the simplex $D_{1}$, we change the factor
\[
  \frac{1}{p!} \left( \log \frac{1-t_{1}}{1-u_{1}} + \log \frac{t_{2}}{t_{1}}
    + \log \frac{u_{2}}{u_{1}} \right)^{p}
\]
to
\[
  \frac{1}{p!} \left( \log \frac{1-t_{1}}{1-t_{2}}
    + \log \frac{1-t_{2}}{1-u_{1}} + \log \frac{t_{2}}{t_{1}}
    + \log \frac{u_{2}}{u_{1}} \right)^{p}
\]
and then expand it as
\[
  \sum_{m+n=p} \sum_{a+b=m} \sum_{c+d=n} \frac{1}{a!b!c!d!}
    \left( \log \frac{1-t_{1}}{1-t_{2}} \right)^{a}
    \left( \log \frac{1-t_{2}}{1-u_{1}} \right)^{b}
    \left( \log \frac{t_{2}}{t_{1}} \right)^{c}
    \left( \log \frac{u_{2}}{u_{1}} \right)^{d}.
\]
Note that the factors
\[
  \frac{1}{a!} \left( \log \frac{1-t_{1}}{1-t_{2}} \right)^{a}
  \quad \textrm{and} \quad
  \frac{1}{c!} \left( \log \frac{t_{2}}{t_{1}} \right)^{c}
\]
form a sum, so the corresponding multiple zeta values are
\[
  \sum_{m+n=p} \sum_{a+b=m} \sum_{c+d=n} \sum_{\mid\boldsymbol{\alpha}\mid=a+c+1}
    \zeta(\alpha_{0}, \alpha_{1}, \ldots, \alpha_{a}+1, \{1\}^{b}, d+2).
\]
The inner sum can be evaluated one by one according to the value of $a+b$. When $a+b = 0$, the sum is
\[
  \sum_{c+d=p} \zeta(c+2,d+2).
\]
It is the sum of whole double zeta values of weight $p+4$ except $\zeta(1,p+3)$. So it is equal, by the sum formula, to
\[
  \zeta(p+4) - \zeta(1,p+3).
\]
When $a+b = 1$, the corresponding sums are
\[
  \sum_{c+d=p-1} \zeta(c+2,1,d+2)
    + \sum_{c_{1}+c_{2}+d=p-1} \zeta(c_{1}+1,c_{2}+2,d+2).
\]
The second sum is equal to
\[
  \sum_{c_{1}+c_{2}+d=p} \zeta(c_{1}+1,c_{2}+1,d+2)
    - \sum_{c+d=p} \zeta(c+1,1,d+2).
\]
After a cancellation with the first term, it is equal to
\[
  \zeta(p+4) - \zeta(1,1,p+2).
\]
Repeating such a procedure, we obtain that
\[
  I(p;1)
  = \sum_{0 \leq m \leq p} \big[ \zeta(p+4) - \zeta(\{1\}^{m+1},p-m+3) \big].
\]

The simplex $D_{2}: 0 < u_{1} < u_{2} < t_{1} < t_{2} < 1$ is obtained from $D_{1}$ by exchanging $t_{1}$, $t_{2}$ with $u_{1}$, $u_{2}$. Therefore, we 
represent $I(p;2)$ by
\[
 \frac{1}{p!} \int_{D_{2}} \left( -\log \frac{1-u_{1}}{1-u_{2}}
    - \log \frac{1-u_{2}}{1-t_{1}} + \log \frac{u_{2}}{u_{1}}
    + \log \frac{t_{2}}{t_{1}} \right)^{p}
    \frac{\mathrm{d}t_{1} \mathrm{d}t_{2}}{(1-t_{1})t_{2}} \frac{\mathrm{d}u_{1} \mathrm{d}u_{2}}{(1-u_{1})u_{2}},
\]
and hence
\[
  I(p;2)
  = \sum_{0 \leq m \leq p} (-1)^{m}
    \big[ \zeta(p+4) - \zeta(\{1\}^{m+1},p-m+3) \big].
\]
Therefore, we get
\begin{equation} \label{e4.3}
  I(p;1)+I(p;2)
  = 2 \sum_{\substack{0 \leq m \leq p \\ m: \textup{ even}}}
    \big[ \zeta(p+4) - \zeta(\{1\}^{m+1},p-m+3) \big].
\end{equation}

By similar construction of the integral \eqref{e4.1} over different simplex, it is easy to see that $I(p;3) = I(p;4)$ and $I(p;5) = I(p;6)$. On the simplex $D_{3}$, the factor
\[
  \frac{1}{p!} \left( \log \frac{1-t_{1}}{1-u_{1}}
    + \log \frac{t_{2}}{t_{1}} + \log \frac{u_{2}}{u_{1}} \right)^{p}
\]
is replaced by
\[
  \frac{1}{p!} \left( \log \frac{1-t_{1}}{1-u_{1}}
    + \log \frac{u_{1}}{t_{1}} + 2\log \frac{t_{2}}{u_{1}}
    + \log \frac{u_{2}}{t_{2}} \right)^{p}
\]
and then expanded as
\[
  \sum_{m+n=p} \sum_{a+b+c=n} \frac{2^{b}}{m!a!b!c!}
    \left( \log \frac{1-t_{1}}{1-u_{1}} \right)^{m}
    \left( \log \frac{u_{1}}{t_{1}} \right)^{a}
    \left( \log \frac{t_{2}}{u_{1}} \right)^{b}
    \left( \log \frac{u_{2}}{t_{2}} \right)^{c}.
\]
In terms of multiple zeta values, we have
\[
  I(p;3)
  = \sum_{m+n=p} \sum_{a+b+c=n} 2^{b} \sum_{\mid\boldsymbol{\alpha}\mid=m+a+1}
    \zeta(\alpha_{0}, \alpha_{1}, \ldots, \alpha_{m}, b+c+3).
\]
Combining $b$, $c$ into a single variable and employing
\[
  \sum_{b+c=0}^{\alpha_{m+1}-2} 2^{b} = 2^{\alpha_{m+1}-1} - 1,
\]
we conclude that
\[
  I(p;3)
  = \sum_{0 \leq m \leq p} \sum_{\mid\boldsymbol{\alpha}\mid=p+3}
    (2^{\alpha_{m+1}-1}-1)
    \zeta(\alpha_{0}, \alpha_{1}, \ldots, \alpha_{m}, \alpha_{m+1}+1).
\]

On the simplex $D_{5}$, the factor in the integral \eqref{e4.1} is replaced by
\[
  \frac{1}{p!} \left( -\log \frac{1-u_{1}}{1-t_{1}}
    + \log \frac{t_{1}}{u_{1}} + 2\log \frac{u_{2}}{t_{1}}
    + \log \frac{t_{2}}{u_{2}} \right)^{p}
\]
and expanded as
\[
  \sum_{m+n=p} (-1)^{m} \sum_{a+b+c=n} \frac{2^{b}}{m!a!b!c!}
    \left( \log \frac{1-u_{1}}{1-t_{1}} \right)^{m}
    \left( \log \frac{t_{1}}{u_{1}} \right)^{a}
    \left( \log \frac{u_{2}}{t_{1}} \right)^{b}
    \left( \log \frac{t_{2}}{u_{2}} \right)^{c}
\]
and
\[
  I(p;5)
  = \sum_{m+n=p} (-1)^{m} \sum_{a+b+c=n} 2^{b}
    \sum_{\mid\boldsymbol{\alpha}\mid=m+a+1}
    \zeta(\alpha_{0}, \alpha_{1}, \ldots, \alpha_{m}, b+c+3),
\]
or
\[
  \sum_{0 \leq m \leq p} (-1)^{m} \sum_{\mid\boldsymbol{\alpha}\mid=p+3}
    (2^{\alpha_{m+1}-1}-1)
    \zeta(\alpha_{0}, \alpha_{1}, \ldots, \alpha_{m}, \alpha_{m+1}+1).
\]

Now $I(p;3)$, $I(p;4)$, $I(p;5)$ and $I(p;6)$ are evaluated in terms of multiple zeta values. Summing them together to yield
\begin{equation} \label{e4.4}
  \sum_{j=3}^{6} I(p;j)
  = 2\sum_{\substack{0 \leq m \leq p \\ m: \textup{ even}}}
    \sum_{\mid\boldsymbol{\alpha}\mid=p+3} 2^{\alpha_{m+1}}
    \zeta(\alpha_{0}, \alpha_{1}, \ldots, \alpha_{m}, \alpha_{m+1}+1)
    - 4\sum_{\substack{0 \leq m \leq p \\ m: \textup{ even}}} \zeta(p+4).
\end{equation}
After an elementary calculation to \eqref{e4.2}, \eqref{e4.3} and \eqref{e4.4}, we complete the proof.
\end{proof}
\section{Convolution of $Z_-(n)$ and $Z^\star_+(n)$} \label{sec.5}
The values of the functions $Z_-(n)$, $Z_+(n)$, and $Z^\star_+(n)$ are strongly related to 
$\zeta(\ov{s})$. Therefore, we will evaluate the value of the finite convolution 
of $Z_-(n)$ and $Z^\star_+(n)$ using their integral representations of alternating zeta values.
We recall Eq.\,(\ref{eq.aeuler}):
$$
\zeta(\ov{p+1})=\frac{-1}{p!}\int^1_0\left(\log\frac1{t}\right)^p\frac{dt}{1+t}.
$$
\begin{theorem}
For any nonnegative integer $p$, we have
$$
\sum_{m+n=p}Z_-(m)Z^\star_+(n)
=2\left((-1)^p-1\right)\zeta(p+2,\ov{2})
+(p+2)\sum_{a+b=p+1}2^{a+2}\zeta(b+1,\ov{a+2}).
$$
\end{theorem}
\begin{proof}
Applying Proposition \ref{P3.1} and Eq.\,(\ref{eq.aeuler}), we have
$$
Z_-(m)=-\bigl( 1+(-1)^m\bigr)\zeta(\ov{m+2})
=\frac{1+(-1)^m}{(m+1)!}\int^1_0\left(\log\frac1{t}\right)^{m+1}\frac{dt}{1+t}.
$$
Using Eq.\,(\ref{eq.d.1}) and Eq.\,(\ref{eq.aeuler}), we have
$$
Z^\star_+(n)=-2(n+1)\zeta(\ov{n+2})
=\frac{2}{n!}\int^1_0\left(\log\frac1{u}\right)^{n+1}\frac{du}{1+u}.
$$
We now carry out the shuffle products of $Z_-(m)$ and $Z^\star_+(n)$ with $m+n=p$.
The value of 
$$
\sum_{m+n=p}Z_-(m)Z^\star_+(n)
$$
is 
$$
\frac2{(p+1)!}\int_{[0,1]\times [0,1]}
K_p(t,u)\frac{dtdu}{(1+t)(1+u)},
$$
where 
$$
K_p(t,u)=\left(\log\frac1{t}+\log\frac1{u}\right)^{p+1}\log\frac1{u}
-\left(\log\frac1{u}-\log\frac1{t}\right)^{p+1}\log\frac1{u}.
$$

We decompose $[0,1]\times [0,1]$ into two disjoint union of dimension $2$:
$$
D_1: 0<t<u<1, \qquad D_2: 0<u<t<1.
$$
For $j=1,2$, we set
$$
I_j=\frac2{(p+1)!}\int_{[0,1]\times [0,1]}
K_p(t,u)\frac{dtdu}{(1+t)(1+u)}.
$$
On the simplex $D_1$, we change the factors
$$
\left(\log\frac1{t}+\log\frac1{u}\right)^{p+1}
=\left(\log\frac{u}{t}+2\log\frac1{u}\right)^{p+1},
$$
and
$$
\left(\log\frac1{u}-\log\frac1{t}\right)^{p+1}
=\left(\log\frac{u}{t}\right)^{p+1}.
$$
Expanding the first term as
$$
\sum_{a+b=p+1}\frac{2^a(p+1)!}{a!b!}\left(\log\frac{u}{t}\right)^b
\left(\log\frac1{u}\right)^a,
$$
we substitute them in the original representation of $I_1$. The corresponding alternating
zeta values are
$$
I_1=\sum_{a+b=p+1}2^{a+1}(a+1)\zeta(b+1,\ov{a+2})
+2(-1)^p\zeta(p+2,\ov{2}).
$$

On the simplex $D_2$, the integral $I_2$ becomes
$$
\frac{2}{(p+1)!}\int_{D_2}\left[
\left(\log\frac{t}{u}+2\log\frac1{t}\right)^{p+1}
-\left(\log\frac{t}{u}\right)^{p+1}\right]
\left(\log\frac{t}{u}+\log\frac1{t}\right)
\frac{dudt}{(1+u)(1+t)}.
$$
Expanding them out, the corresponding alternating zeta values are
$$
I_2=\sum_{a+b=p+1}2^{a+1}\left[(b+1)\zeta(b+2,\ov{a+1})
+(a+1)\zeta(b+1,\ov{a+2})\right] 
-2(p+2)\zeta(p+3,\ov{1})-2\zeta(p+2,\ov{2}).
$$

After an elementary calculation to $I_1+I_2$, we have
$$
\sum_{m+n=p}Z_-(m)Z^\star_+(n)
=2\left((-1)^p-1\right)\zeta(p+2,\ov{2})
+(p+2)\sum_{a+b=p+1}2^{a+2}\zeta(b+1,\ov{a+2}).
$$
\end{proof}

We let $p=0$ in Eq.\,(\ref{eq.2barm}) and set $q=0$ in Eq.\,(\ref{eq.1barm}),
then we can get two sum formulas of the alternating Euler double sums.
\begin{proposition} \cite{Teo2018}
Given a nonnegative integer $n$, we have
\begin{align}
\sum_{a_1+a_2=n}\zeta(a_1+1,\ov{a_2+1})
&=\zeta(1,\ov{n+1})-\zeta(\ov{1})\zeta(\ov{n+1}),\\
\sum_{a_1+a_2=n+1}\zeta(\ov{a_1+1},\ov{a_2+1})
&=\zeta(\ov{1})\zeta(n+2)-\zeta(\ov{1},n+2).\label{eq.5.2}
\end{align}
\end{proposition}

We now get the weighted sum formula for an alternating Euler double sums.
\begin{theorem}
For a given nonnegative integer $n$, we get
\begin{equation}
\sum_{a+b=n}2^{b+1}\zeta(a+1,\ov{b+1})
=(n+1)\zeta(n+2)+2\zeta(n+1,\ov{1})+2\zeta(\ov{n+2}).
\end{equation}
\end{theorem}
\begin{proof}
Eq.\,(\ref{eq.2w}) gives
$$
\sum_{a+b=n}2^{b+1}\zeta(a+1,\ov{b+1})
=\sum_{a+b=n}\zeta(\ov{a+1})\zeta(\ov{b+1}).
$$
We use the stuffle relation Eq.\,(\ref{eq.2.12}) and get
$$
\sum_{a+b=n}\zeta(\ov{a+1})\zeta(\ov{b+1})
=(n+1)\zeta(n+2)+2\sum_{a+b=n}\zeta(\ov{a+1},\ov{b+1}).
$$
This sum is evaluated in Eq\,(\ref{eq.5.2}). Therefore, we apply Eq.\,(\ref{eq.2.11})
and get the desired formula.
\end{proof}

Now the formula of the finite convolution of $Z_-(n)$ and $Z^\star_+(n)$ can be evaluated 
more direct.
\begin{corollary}
For any nonnegative integer $p$, we have
\begin{equation}\label{eq.54}
\sum_{m+n=p}Z_-(m)Z^\star_+(n)
=2\left((-1)^p-1\right)\zeta(p+2,\ov{2})
+(p+2)(p+1-2^{-p-2})\zeta(p+4).
\end{equation}
\end{corollary}
Let $p$ be an even nonnegative integer. Eq.\,(\ref{eq.54}) now becomes Eq.\,(\ref{eq.adw}):
$$
\sum^p_{m=0}Z_-(2m)Z^\star_+(2p-2m)
=(2p+2)(2p+1+2^{-(p+2)})\zeta(2p+4).
$$
\section{Weighted Euler sum formula} \label{sec.6}
The weighted Euler sum formula
\[
  \sum_{\alpha_{1}+\alpha_{2}=n} 2^{\alpha_{2}} \zeta(\alpha_{1},\alpha_{2}+1) = \frac{n+2}{2} \zeta(n+1)
\]
was proved by Ohno and Zudilin \cite{OZ08} in 2008. Such a formula indeed is a consequence of Euler decomposition formula
\[
  \zeta(p+2) \zeta(q+2)
  = \sum_{\alpha_{1}+\alpha_{2}=p+q+3}
    \left[ \binom{\alpha_{2}}{p+1} + \binom{\alpha_{2}}{q+1} \right] \zeta(\alpha_{1},\alpha_{2}+1)
\]
as we have mentioned in Proposition~\ref{prop.27}. Now we shall produce a more general weighted sum formula. Of course, it covers the above.

\begin{theorem}
Let $p$ be a nonnegative integer and $\lambda$ a real number with $\lambda \neq 1$. Then
\begin{equation}
  \sum_{\alpha_{1}+\alpha_{2}=p+3}
    \big[ (\lambda^{\alpha_{1}-1}+1) (\lambda+1)^{\alpha_{2}}
    - \lambda^{\alpha_{1}-1} - \lambda^{\alpha_{2}} \big] \zeta(\alpha_{1},\alpha_{2}+1)
  = \frac{\lambda^{p+3}-1}{\lambda-1} \zeta(p+4).
\end{equation}
\end{theorem}
\begin{proof}
We begin with the integral
\begin{equation} \label{e5.1}
  J(p)
  = \frac{1}{p!} \int_{E_{2} \times E_{2}}
    \left( \log \frac{t_{2}}{t_{1}} + \lambda \log \frac{u_{2}}{u_{1}} \right)^{p}
    \frac{\mathrm{d}t_{1} \mathrm{d}t_{2}}{(1-t_{1})t_{2}} \frac{\mathrm{d}u_{1} \mathrm{d}u_{2}}{(1-u_{1})u_{2}}.
\end{equation}
Expanding the integrand as
\[
  \sum_{m+n=p} \frac{\lambda^{n}}{m!n!}
    \left( \log \frac{t_{2}}{t_{1}} \right)^{m}
    \left( \log \frac{u_{2}}{u_{1}} \right)^{n}
\]
and noting that
\[
  \frac{1}{r!} \int_{E_{2}} \left( \log \frac{t_{2}}{t_{1}} \right)^{r}
    \frac{\mathrm{d}t_{1} \mathrm{d}t_{2}}{(1-t_{1})t_{2}}
  = \zeta(r+2),
\]
we obtain immediately that
\begin{equation} \label{e5.2}
  J(p) = \sum_{m+n=p} \lambda^{n} \zeta(m+2) \zeta(n+2).
\end{equation}

As a replacement of shuffle product, we decompose $E_{2} \times E_{2}$ into a disjoint union of six simplices of dimension $4$:
\begin{gather*}
  D_{1}: 0 < t_{1} < t_{2} < u_{1} < u_{2} < 1, \quad
  D_{2}: 0 < u_{1} < u_{2} < t_{1} < t_{2} < 1, \\
  D_{3}: 0 < t_{1} < u_{1} < t_{2} < u_{2} < 1, \quad
  D_{4}: 0 < t_{1} < u_{1} < u_{2} < t_{2} < 1, \\
  D_{5}: 0 < u_{1} < t_{1} < u_{2} < t_{2} < 1 \quad \textrm{and} \quad
  D_{6}: 0 < u_{1} < t_{1} < t_{2} < u_{2} < 1.
\end{gather*}
For $j = 1,2,\ldots,6$, we let
\[
  J(p;j)
  = \frac{1}{p!} \int_{D_{j}}
    \left( \log \frac{t_{2}}{t_{1}} + \lambda \log \frac{u_{2}}{u_{1}} \right)^{p} \frac{\mathrm{d}t_{1} \mathrm{d}t_{2}}{(1-t_{1})t_{2}} \frac{\mathrm{d}u_{1} \mathrm{d}u_{2}}{(1-u_{1})u_{2}}.
\]
Of course, we have
\[
  J(p) = \sum_{j=1}^{6} J(p;j).
\]

Next, we are going to evaluate each $J(p;j)$ in terms of multiple zeta values. On the simplex $D_{1}$, the integrand
\[
  \frac{1}{p!} \left( \log \frac{t_{2}}{t_{1}}
    + \lambda \log \frac{u_{2}}{u_{1}} \right)^{p}
\]
is expanded as
\[
  \sum_{m+n=p} \frac{\lambda^{n}}{m!n!}
    \left( \log \frac{t_{2}}{t_{1}} \right)^{m}
    \left( \log \frac{u_{2}}{u_{1}} \right)^{n},
\]
so the value is
\[
  J(p;1) = \sum_{m+n=p} \lambda^{n} \zeta(m+2,n+2).
\]
With the same expansion for the integrand, we get
\[
  J(p;2) = \sum_{m+n=p} \lambda^{m} \zeta(m+2,n+2).
\]
Employing the reflection formula
\[
  \zeta(u,v) + \zeta(v,u) = \zeta(u) \zeta(v) - \zeta(u+v), \quad u,v \geq 2
\]
we conclude that
\begin{equation} \label{e5.3}
  J(p;1)+J(p;2)
  = \sum_{m+n=p} \lambda^{n} \zeta(m+2) \zeta(n+2)
    - \frac{\lambda^{p+1}-1}{\lambda-1} \zeta(p+4).
\end{equation}

On the simplex $D_{3}$, the integrand is replaced by
\[
  \frac{1}{p!} \left( \log \frac{u_{1}}{t_{1}}
    + (\lambda+1) \log \frac{t_{2}}{u_{1}}
    + \lambda \log \frac{u_{2}}{t_{2}} \right)^{p}
\]
and expanded as
\[
  \sum_{a+b+c=p} \frac{(\lambda+1)^{b} \lambda^{c}}{a!b!c!}
    \left( \log \frac{u_{1}}{t_{1}} \right)^{a}
    \left( \log \frac{t_{2}}{u_{1}} \right)^{b}
    \left( \log \frac{u_{2}}{t_{2}} \right)^{c}
\]
so that we have
\[
  J(p;3) = \sum_{a+b+c=p} (\lambda+1)^{b} \lambda^{c} \zeta(a+1,b+c+3),
\]
or
\[
  \sum_{\alpha_{1}+\alpha_{2}=p+3}
    \big[ (\lambda+1)^{\alpha_{2}-1} - \lambda^{\alpha_{2}-1} \big] \zeta(\alpha_{1},\alpha_{2}+1).
\]

On the simplex $D_{4}$, the integrand is replaced by
\[
  \frac{1}{p!} \left( \log \frac{u_{1}}{t_{1}}
    + (\lambda+1) \log \frac{u_{2}}{u_{1}} + \log \frac{t_{2}}{u_{2}} \right)^{p},
\]
and hence
\[
  J(p;4)
  = \sum_{\alpha_{1}+\alpha_{2}=p+3} \frac{1}{\lambda}
    \big[ (\lambda+1)^{\alpha_{2}-1} - 1 \big] \zeta(\alpha_{1},\alpha_{2}+1).
\]
Combining $J(p;3)$ and $J(p;4)$ together to obtain
\begin{equation} \label{e5.4}
  J(p;3)+J(p;4)
  = \frac{1}{\lambda} \sum_{\alpha_{1}+\alpha_{2}=p+3}
    \big[ (\lambda+1)^{\alpha_{2}} - \lambda^{\alpha_{2}} \big] \zeta(\alpha_{1},\alpha_{2}+1) - \frac{1}{\lambda} \zeta(p+4).
\end{equation}

On the simplex $D_{5}$, the integrand is replaced by
\[
  \frac{1}{p!} \left( \lambda \log \frac{t_{1}}{u_{1}}
    + (\lambda+1) \log \frac{u_{2}}{t_{1}} + \log \frac{t_{2}}{u_{2}} \right)^{p}
\]
so that
\begin{align*}
  J(p;5)
  &= \sum_{a+b+c=p} \lambda^{a} (\lambda+1)^{b} \zeta(a+1,b+c+3) \\
  &= \lambda^{p} \sum_{a+b+c=p} \left( \frac{\lambda+1}{\lambda} \right)^{b}
    \left( \frac{1}{\lambda} \right)^{c} \zeta(a+1,b+c+3),
\end{align*}
or
\[
  \lambda^{p} \sum_{\alpha_{1}+\alpha_{2}=p+3}
    \left[ \left( \frac{\lambda+1}{\lambda} \right)^{\alpha_{2}-1}
    - \left( \frac{1}{\lambda} \right)^{\alpha_{2}-1} \right] \zeta(\alpha_{1},\alpha_{2}+1).
\]
In the similar way, we get
\[
  J(p;6)
  = \sum_{a+b+c=p} \lambda^{a+c} (\lambda+1)^{b} \zeta(a+1,b+c+3),
\]
or
\[
  \lambda^{p+1} \sum_{\alpha_{1}+\alpha_{2}=p+3}
    \left[ \left( \frac{\lambda+1}{\lambda} \right)^{\alpha_{2}-1} - 1 \right] \zeta(\alpha_{1},\alpha_{2}+1).
\]
Therefore, we have
\begin{equation} \label{e5.5}
  J(p;5)+J(p;6)
  = \lambda^{p+1} \sum_{\alpha_{1}+\alpha_{2}=p+3}
    \left[ \left( \frac{\lambda+1}{\lambda} \right)^{\alpha_{2}}
    - \left( \frac{1}{\lambda} \right)^{\alpha_{2}} \right] \zeta(\alpha_{1},\alpha_{2}+1) - \lambda^{p+1} \zeta(p+4).
\end{equation}
Our assertion then follows from a calculation to the relations \eqref{e5.2}, \eqref{e5.3}, \eqref{e5.4} and \eqref{e5.5}.
\end{proof}

\begin{remark}
Here is another kind of weighted sum formula on multiple zeta values 
can be found in \cite[Main Theorem]{OEL13}:
For a pair of positive integers $n$ and $k$ with $n \geq k$ and $k$ even, we have
\[
  \sum_{\mid\boldsymbol{\alpha}\mid=n}
    (2^{\alpha_{2}} + 2^{\alpha_{4}} + \cdots + 2^{\alpha_{k}})
    \zeta(\alpha_{1}, \alpha_{2}, \ldots, \alpha_{k-1}, \alpha_{k}+1)
  = \frac{n+k}{2} \zeta(n+1).
\]
The above theorem was obtained from shuffle product with the double integral
\[
  \frac{1}{k!r!} \int_{E_{2} \times E_{2}}
    \left( \log \frac{1-t_{1}}{1-t_{2}} - \log \frac{1-u_{1}}{1-u_{2}} \right)^{k}
    \left( \log \frac{t_{2}}{t_{1}} 
    + \log \frac{u_{2}}{u_{1}} \right)^{r} 
    \frac{\mathrm{d}t_{1} \mathrm{d}t_{2}}{(1-t_{1})t_{2}} 
    \frac{\mathrm{d}u_{1} \mathrm{d}u_{2}}{(1-u_{1})u_{2}}.
\]
\end{remark}



\begin{thebibliography}{99}
\bibitem{AKO2007}
Aoki, T., Kombu, Y., Ohno, Y.:
\emph{A generating function for sums of multiple zeta values
and its applications},
Proc. Amer. Math. Soc.
\textbf{136}, no. 2, 387--395 (2007).
\href{https://doi.org/10.1090/S0002-9939-07-09175-7}%
{https://doi.org/10.1090/S0002-9939-07-09175-7}.

\bibitem{AK1999}
Arakawa, T., Kanekno, M.:
\emph{Multiple zeta values, poly-Bernoulli numbers, and related zeta functions},
Nagoya Math. J. 
\textbf{153}, 189--209 (1999).

\bibitem{BBB1997}
Borwein, J.~M., Bradley, D.~M., Broadhurst, D.~J.:
\emph{Evaluations of $k$-fold Euler/Zagier sums: 
a compendium of results for arbitrary $k$}, 
Electron. J. Combin. \textbf{4}, no.~2, Research Paper 5. approx. 21 pp. (1997).

\bibitem{Chen2017}
Chen, K.-W.:
\emph{Generalized harmonic numbers and Euler sums},
Int. J. Number Theory {\bf 13}, no. 2, 513--528 (2017).
\href{https://doi.org/10.1142/S1793042116500883}%
{https://doi.org/10.1142/S1793042116500883}.

\bibitem{Chen2019}
Chen, K.-W.:
\emph{Generalized Arakawa-Kaneko zeta functions},
Integral Transforms Spec. Funct. {\bf 30}, no. 4, 282--300 (2019).
\href{https://doi.org/10.1080/10652469.2018.1562450}%
{https://doi.org/10.1080/10652469.2018.1562450}.

\bibitem{CE2018}
Chen, K.-W., Eie, M.:
\emph{Some special Euler sums and $\zeta^{\star}(r+2,\{2\}^{n})$}, 
arXiv:1810.11795[math.NT], pages 17, (2018).

\bibitem{CC2010}
Coppo, M.-A., Candelpergher, B.:
\emph{The Arakawa-Kaneko zeta function},
Ramanujan J. {\bf 22}, no. 2, 153--162 (2010).
\href{https://doi.org/10.1007/s11139-009-9205-x}%
{https://doi.org/10.1007/s11139-009-9205-x}.

\bibitem{E09}
Eie, M.:
\emph{Topics in Number Theory}, Monographs in Number Theory \textbf{2}, 
World Scientific, Singapore (2009).
\href{https://doi.org/10.1142/7036}{https://doi.org/10.1142/7036}.

\bibitem{E13}
Eie, M.: \emph{The Theory of Multiple Zeta Values with Applications in Combinatorics},
Monographs in Number Theory \textbf{7}, World Scientific, Singapore (2013). 
\href{https://doi.org/10.1142/8769}{https://doi.org/10.1142/8769}.

\bibitem{ELO09}
Eie, M., Liaw, W.-C., Ong, Y.~L.:
\emph{A restricted sum formula among multiple zeta values}, 
J. Number Theory \textbf{129}, no.~4, 908--921 (2009).
\href{https://doi.org/10.1016/j.jnt.2008.07.012}{https://doi.org/10.1016/j.jnt.2008.07.012}.

\bibitem{EW13}
Eie, M., Wei, C.-S.:
\emph{Generalizations of Euler decomposition and theirapplications}, 
J. Number Theory \textbf{133}, 2475--2495 (2013).
\href{https://doi.org/10.1016/j.jnt.2013.01.010}{https://doi.org/10.1016/j.jnt.2013.01.010}.

\bibitem{E1775}
Euler, L.:
\emph{Meditationes circa singulare serierum genus}, 
Novi Comm. Acad. Sci. Petropol. \textbf{20}, 140--186  (1775). 
[Reprinted in \emph{Opera Omnia}, Ser. I, no.~15, 217--267, Teubner, Berlin (1927).]

\bibitem{G96}
Granville, A.:
\emph{A decomposition of Riemann's zeta-function}, 
in: \emph{Analytic Number Theory} (Kyoto, 1996), 95--102, 
London Math. Soc. Lecture Notes Ser. \textbf{247}, 
Cambridge Univ. Press, Cambridge (1997).
\href{https://doi.org/10.1017/cbo9780511666179.009}%
{https://doi.org/10.1017/cbo9780511666179.009}.

\bibitem{KS2016}
Kaneko, M., Sakata, M.:
\emph{On multiple zeta values of extremal height},
Bull. Aust. Math. Soc. \textbf{93}, 186--193 (2016).
\href{https://doi.org/10.1017/S0004972715001227}%
{https://doi.org/10.1017/S0004972715001227}.

\bibitem{Kargein2020}
Karg{\i}n, L.:
\emph{Poly-p-Bernoulli polynomials and generalized Arakawa-Kaneko zeta function},
Lith. Math. J. \textbf{60}, 29--50 (2020).
\href{https://doi.org/10.1007/s10986-019-09448-7}%
{https://doi.org/10.1007/s10986-019-09448-7}.

\bibitem{LM95}
Le, T.~Q.~T., Murakami, J.:
\emph{Kontsevich's integrals for the Homfly polynomial and relations 
between values of multiple zeta functions}, 
Topology Appl. \textbf{62}, no.~2, 193--206 (1995).
\href{https://doi.org/10.1016/0166-8641(94)00054-7}%
{https://doi.org/10.1016/0166-8641(94)00054-7}.

\bibitem{Ohno2005}
Ohno, Y.:
\emph{Sum relations for multiple zeta values},
In: Aoki, T., Kanemitsu S., Nakahara M., Ohno Y. (eds)
Zeta Functions, Topology and Quantum Physics, Developments in Mathematics,
\textbf{14}, Springer, Boston, MA (2005).
\href{https://doi.org/10.1007/0-387-24981-8_8}{https://doi.org/10.1007/0-387-24981-$8_8$}.

\bibitem{OZ01}
Ohno, Y., Zagier, D.:
\emph{Multiple zeta values of fixed weight, depth, and height}, 
Indag. Math. (N.S.) \textbf{12}, no.~4, 483--487 (2001).
\href{https://doi.org/10.1016/S0019-3577(01)80037-9}%
{https://doi.org/10.1016/S0019-3577(01)80037-9}.

\bibitem{OZ08}
Ohno, Y., Zudilin, W.:
\emph{Zeta stars}, 
Commun. Number Theory Phys. \textbf{2}, no.~2, 325--347 (2008).
\href{https://doi.org/10.4310/CNTP.2008.v2.n2.a2}%
{https://doi.org/10.4310/CNTP.2008.v2.n2.a2}.

\bibitem{OEL13}
Ong, Y.~L., Eie, M., Liaw, W.-C.:
\emph{On generalizations of weighted sum formulas of multiple zeta values}, 
Int. J. Number Theory \textbf{9}, no.~5, 1185--1198 (2013). 
\href{https://doi.org/10.1142/S179304211350019X}%
{https://doi.org/10.1142/S179304211350019X}.

\bibitem{Teo2018}
Teo, L.-P.:
\emph{Alternating double Euler sums, hypergeometric identities and a theorem of Zagier},
J. Math. Anal. Appl. \textbf{462}, 777--800 (2018).
\href{https://doi.org/10.1016/j.jmaa.2018.02.037}%
{https://doi.org/10.1016/j.jmaa.2018.02.037}.

\bibitem{Xu2019}
Xu, C.:
\emph{Integrals of logarithmic functions and alternating multiple zeta values},
Math. Slovaca \textbf{69}, no. 2, 339--356 (2019).
\href{https://doi.org/10.1515/ms-2017-0227}{https://doi.org/10.1515/ms-2017-0227}.

\end{thebibliography}
\end{document}